\documentclass[12pt]{amsart}

\usepackage{amsfonts, amsthm, amsmath, amssymb}

\usepackage{rotating}
\usepackage{tikz}

\usepackage{amscd}

\usepackage[latin2]{inputenc}

\usepackage{t1enc}

\usepackage[mathscr]{eucal}

\usepackage{indentfirst}

\usepackage{graphicx}

\usepackage{graphics}

\usepackage{pict2e}

\usepackage{mathrsfs}

\usepackage{enumerate}
\usepackage[pagebackref]{hyperref}
\usepackage{cite}
\usepackage{color}
\usepackage{epic}
\usepackage{hyperref} %this gives clickable references, which is nice
\numberwithin{equation}{section}
\def\red{\textcolor{red}}
\theoremstyle{plain}

\newtheorem{theorem}{Theorem}[section]

\newtheorem{lemma}[theorem]{Lemma}

\newtheorem{corollary}[theorem]{Corollary}

\newtheorem{proposition}[theorem]{Proposition}

\theoremstyle{definition}

\newtheorem{Def}[theorem]{Definition}

\newtheorem{conjecture}[theorem]{Conjecture}

\newtheorem{remark}[theorem]{Remark}

\newtheorem{?}[theorem]{Problem}

\newcommand{\N}{\mathbb{N}}

\newcommand{\Z}{\mathbb{Z}}

\def\DES{\mathrm{DES}}

\def\max{\mathrm{max}}
\def\min{\mathrm{min}}

\def\dd{\mathrm{dd}}
\def\des{\mathrm{des}}

\def\S{\mathfrak{S}}
\def\D{\mathfrak{D}}
\def\T{\mathfrak{T}}

\def\E{\mathcal{E}}

\def\Orb{\mathrm{Orb}}

\def\ides{\mathrm{ides}}

\def\boxit#1{\leavevmode\hbox{\vrule\vtop{\vbox{\kern.33333pt\hrule
    \kern1pt\hbox{\kern1pt\vbox{#1}\kern1pt}}\kern1pt\hrule}\vrule}}

\usepackage{collectbox}



\begin{document}

\title[Unimodal descent polynomials]{On two unimodal descent polynomials}

\author[S. Fu]{Shishuo Fu}
\address[Shishuo Fu]{College of Mathematics and Statistics, Chongqing University, Chongqing 401331, P.R. China} 
\email{fsshuo@cqu.edu.cn}

\author[Z. Lin]{Zhicong Lin}
\address[Zhicong Lin]{School of Science, Jimei University, Xiamen 361021,
P.R. China
\& CAMP, National Institute for Mathematical Sciences, Daejeon 305-811, Republic of Korea} 
\email{lin@nims.re.kr}

\author[J. Zeng]{Jiang Zeng}

\address[Jiang Zeng]{Institut Camille Jordan,  Universit\'{e} Claude Bernard Lyon 1, France}
\email{zeng@math.univ-lyon1.fr}
  
\date{\today}

\begin{abstract} 
The descent polynomials of separable permutations and derangements are 
both  demonstrated to be  
unimodal. Moreover, 
we prove that the $\gamma$-coefficients of 
the first   are positive with
an  interpretation parallel 
 to the classical Eulerian polynomial, while
the second  is spiral, a  property stronger than unimodality. 
Furthermore, we conjecture that they are both real-rooted.
\end{abstract}

\maketitle

%\tableofcontents

\section{Introduction}
 Many polynomials with combinatorial meanings have been shown to be unimodal; 
see the recent survey of Br\"and\'en~\cite{bran}. Recall that a polynomial $h(t)=\sum_{i=0}^dh_it^i$ of degree $d$ is said to be {\em unimodal} if the coefficients are increasing 
and then decreasing, i.e., there is an index $c$ such that
$
h_0\leq h_1\leq \cdots\leq h_c\geq h_{c+1}\geq\cdots\geq h_d.
$
Let $p(t)=a_rt^r+a_{r+1}t^{r+1}+\cdots +a_{s} t^{s}$ be a real polynomial with $a_r\neq0$ and $a_s\neq0$. It is called  \emph{palindromic}   (or \emph{symmetric}) {\em of center $n/2$}  if $n=r+s$ and $a_{r+i}=a_{s-i}$ for $0\leq i\leq n/2$. For example,
 polynomials $1+t$ and $t$ are palindromic of center  $1/2$ and 1, respectively.
Any palindromic  polynomial $p(t)\in \Z[t]$  can be written uniquely~\cite{bran,swz} as
  $$
p(t)=\sum_{k=r}^{\lfloor\frac{n}{2}\rfloor}\gamma_{k}t^k(1+t)^{n-2k},
$$
where $\gamma_{k}\in\Z$. If $\gamma_{k}\geq 0$ then we say that it is 
 {\em $\gamma$-positive of  center $n/2$}.
It is clear  that the 
$\gamma$-positivity implies palindromicity and unimodality. 
Three prototypes  of combinatorial $\gamma$-positive  polynomials are the binomial polynomials $(1+x)^n$ with $n\in \N$, Eulerian polynomials and Narayana polynomials; see~\eqref{Eulerian} and~\eqref{Nara} below. 
For further $\gamma$-positivity results and problems, the reader is referred to the excellent exposition by Petersen~\cite{Pet} and  the most recent survey
 by Athanasiadis~\cite{Ath}.
The aim of this paper is to provide two new families of combinatorial unimodal 
polynomials,  of which one is $\gamma$-positive (Theorem~\ref{gamma:s}) and another is not palindromic but  has spiral property, which also implies  the unimodality (Theorem~\ref{thm:spiral}).

Let  $\S_n$ be the set of all permutations of $[n]:=\{1,2,\ldots,n\}$. For a permutation $\pi\in\S_n$, written as 
$\pi=\pi_1\pi_2\ldots \pi_n$,
an index $i\in [n]$ is a {\em descent} (resp.~\emph{double descent}) 
of $\pi$ if $\pi_i>\pi_{i+1}$ (resp.~$\pi_{i-1}>\pi_i>\pi_{i+1}$), where 
$\pi_0=\pi_{n+1}=+\infty$. 
Denote by $\des(\pi)$ and $\dd(\pi)$ the number of descents and double descents of $\pi$, respectively.
It is known~\cite{fs, Pet} (see also~\cite{linzen15})  that   the 
descent polynomial on $\S_n$  is the $n$-th {\em Eulerian polynomial}, which is $\gamma$-positive of center  $(n-1)/2$:
\begin{align}\label{Eulerian}
A_n(t):=\sum_{\pi\in\S_n}t^{\des(\pi)}=\sum_{k=0}^{\lfloor\frac{n-1}{2}\rfloor}\gamma_{n,k}^A t^k(1+t)^{n-1-2k},
\end{align}
 where  $\gamma_{n,k}^A=\#\{\pi\in \S_n: \dd(\pi)=0,  \des(\pi)=k\}$.

%The main objective of this paper is to prove that the descent polynomials on two interesting subsets of symmetric groups, namely, separable permutations and derangements, are unimodal ($\gamma$-positive in the case of separable permutations). 

Patterns in permutations has been extensively studied in the literature (see for instance
Kitaev's book~\cite{ki}). 
A permutation $\pi$ is said to contain the permutation $\sigma$  if there exists a subsequence of (not necessarily consecutive) entries of $\pi$ that has the same relative order as $\sigma$, and in this case $\sigma$ is said to be a pattern of $\pi$; otherwise, 
$\pi$ is said to avoid  $\sigma$.  The set of permutations avoiding patterns
$\sigma_1, \ldots, \sigma_r$ in $\S_n$ is denoted by $\S_n(\sigma_1, \ldots, \sigma_r)$.
The descent polynomial over $\S_n(231)$ is the $n$-th 
 {\em Narayana polynomial}~\cite[Chapter 2]{Pet},  which 
 is also $\gamma$-positive of center $(n-1)/2$; see \cite[Theorem 4.2]{Pet} or \cite[Proposition 11.14]{prw} for an equivalent statement:
\begin{align}\label{Nara}
N_n(t):=\sum_{\pi\in \S_n(231)} t^{\des(\pi)}=\sum_{k=0}^{\lfloor\frac{n-1}{2}\rfloor}\gamma_{n,k}^N t^k(1+t)^{n-1-2k},
\end{align}
 where  $\gamma_{n,k}^N=\#\{\pi\in \S_n(231): \dd(\pi)=0,  \des(\pi)=k\}$. 
 
A  permutation avoiding patterns $2413$ and $3142$ is called a 
 \emph{separable permutation}. 
It is known (see \cite{sh,we}) that separable permutations are counted by the \emph{large Schr\"oder numbers}.  The first few numbers are $1,2,6,22,90,394,1806$, see \href{http://oeis.org/A006318}{\tt{oeis:A006318}}. Our first main result is the following $\gamma$-expansion for the descent polynomial on separable permutations.

\begin{theorem}\label{gamma:s} We have 
\begin{equation}\label{gam:scho}
S_n(t):=\sum_{\pi\in\S_n(2413,3142)}t^{\des(\pi)} =\sum_{k\geq 0}^{\lfloor\frac{n-1}{2}\rfloor}\gamma_{n,k}^S t^k (1+t)^{n-1-2k},
\end{equation}
where 
%\begin{align}\label{gamma:sep1}
$\gamma_{n,k}^S=\#\{\pi \in \S_n(2413, 3142): \dd(\pi)=0, \des(\pi)=k\}$.
%\end{align}
Consequently, the polynomial $S_n(t)$ is $\gamma$-positive and a fortiori, palindromic and unimodal.
\end{theorem}

For example, the first expansions  of $S_n(t)$  read as follows:
\begin{align*}
&S_1(t)=1; \,\,S_2(t)=1+t;\\
&S_3(t)=1+4t+t^2=(1+t)^2+2t;\\
&S_4(t)=1+10t+10t^2+t^3=(1+t)^3+7t(1+t);\\
&S_5(t)=1+20t+48t^2+20t^3+t^4=(1+t)^4+16t(1+t)^2+10t^2;\\
&S_6(t)=1+35t+161t^2+161t^3+35t^4+t^5=(1+t)^5+30t(1+t)^3+61t^2(1+t).
\end{align*}
The palindromicity $S_n(t)=t^{n-1}S_n(1/t)$ follows from the involution
$$
\pi_1\pi_2\cdots\pi_n\mapsto\pi_n\pi_{n-1}\cdots\pi_1
$$
and the fact that $\S_n(2413,3142)$ is invariant under this involution. Though this class of palindromic polynomials already exists in OEIS 
(see~\href{https://oeis.org/A175124}{\tt{oeis:A175124}}), its interpretation as descent polynomials of separable permutations seems new. 
Note that  both~\eqref{Eulerian} and~\eqref{Nara} can be proved  using the {\em modified Foata--Strehl action} (see~\cite{fsh,br,linzen15}) on  $\S_n$, but since $\S_n(2413,3142)$ is not invariant under this action,  it is unclear how Theorem~\ref{gamma:s} could be deduced by the same manner.

%\subsection{The descent polynomial on derangements}
A {\em derangement} is a fixed-point free permutation. Let $\D_n$ be the set of derangements in $\S_n$ and 
consider   the descent polynomial of derangements 
$$
D_n(t):=\sum_{\pi\in\D_n}t^{\des(\pi)}.
$$
The first few values of $D_n(t)$ are listed as follows: 
\begin{align*}
&D_2(t)=t; \,\,D_3(t)=2t; \\
&D_4(t)=4t+4t^2+t^3;\\
&D_5(t)=8t+24t^2+12t^3;\\
&D_6(t)=16t+104t^2+120t^3+24t^4+t^5;\\
&D_7(t)=32t+392t^2+896t^3+480t^4+54t^5.
\end{align*}

We have the following spiral property for $D_n(t)$.

 \begin{theorem}\label{thm:spiral}
 Let $D_n(t)=\sum_{k\geq1}d_{n,k}t^k$. Then, for $n\geq1$ and $1\leq k\leq n-1$
 \begin{equation}\label{spiral}
 d_{2n,2n-k}< d_{2n,k}<d_{2n,2n-k-1}\quad\text{and}\quad
 d_{2n+1,k}< d_{2n+1,2n-k}< d_{2n+1,k+1}
 \end{equation}
  except that $d_{4,1}=d_{4,2}=4$. In particular, the polynomial $D_n(t)$ is unimodal and the maximum coefficient is reached 
  at the center $\lfloor n/2\rfloor$.
 \end{theorem}
\begin{remark}
This kind of spiral property was first observed by Zhang~\cite{zh} for 
the {\em excedance polynomial} on derangements.
 A type B analogue of Zhang's result  was later proved  by Chen, Tang and Zhao~\cite{ctz}  and further generalized by Shin and Zeng~\cite[Corollary 7]{sz} to colored derangements, i.e., derangements 
 in wreath product $\Z_r\rtimes \S_n$ for $r\geq 2$.
 \end{remark}

The rest of this paper is organized as follows. In Section~\ref{sec:spswdt}, we build a simple bijection between separable permutations and some rooted binary trees that we shall call direct-skew trees (or di-sk trees).  Utilizing the model of di-sk trees and a crucial bijection, a proof of Theorem~\ref{gamma:s} is given in Section~\ref{sec:gammasep}.  
We prove Theorem~\ref{thm:spiral} in Section~\ref{sec:spiral}. Finally, we conclude the paper with some remarks and conjectures.

%%%%%%%%%%%%%%%%%%%%%%%%%%%%%%%%%%
\section{From separable permutations to di-sk trees}\label{sec:spswdt}
%%%%%%%%%%%%%%%%%%%%%%%%%%%%%%%%%%

% \begin{figure}
%\begin{tikzpicture}[scale=0.8]
%\draw (2,0) grid (3,1);
%\draw (3,1) grid (4,2);
%\draw (9,1) grid (10,2);
%\draw (10,0) grid (11,1);
%\draw (2.5,0.5) node{$\pi$};
%\draw (9.5,1.5) node{$\pi$};
%\draw (3.5,1.5) node{$\sigma$};
%\draw (10.5,0.5) node{$\sigma$};
%\draw (1,1) node{$\pi\oplus\sigma=$};
%\draw (8,1) node{$\pi\ominus\sigma=$};
%\end{tikzpicture}
%\caption{The direct sum and skew sum operations.}
%\label{twosum}
%\end{figure}
We begin with a well-known characterization of $\S_n(2413,3142)$ using two classical operations on permutations. Let   $\pi\in\S_k$ and $\sigma\in\S_l$ be two permutations. 
The \emph{direct sum} $\pi\oplus\sigma$ and the \emph{skew sum} $\pi\ominus\sigma$, of $\pi$ and $\sigma$, are permutations in $\S_{k+l}$ defined respectively as
$$
(\pi\oplus\sigma)_i=
\begin{cases}
\pi_i, &\text{for $i\in[1,k]$};\\
\sigma_{i-k}+k, &\text{for $i\in[k+1,k+l]$}.
\end{cases}
$$
and
$$
(\pi\ominus\sigma)_i=
\begin{cases}
\pi_i+l, &\text{for $i\in[1,k]$};\\
\sigma_{i-k}, &\text{for $i\in[k+1,k+l]$}.
\end{cases}
$$
 For instance, we have $123\oplus 21=12354$ and $123\ominus 21=34521$. The following  characterization of separable permutations is  folkloric  (cf.~\cite[page 57]{ki}) in pattern avoidance. 
%% If we use permutation matrix to represent permutations, then the direct sum and the skew sum are forming block anti-diagonal matrix and block diagonal matrix, respectively (see Fig.~\ref{twosum}).

 \begin{proposition}\label{desides}
 A permutation is separable if and only if it can be built from the permutation $1$ by applying the operations $\oplus$ and $\ominus$ repeatedly. 
 \end{proposition}

Next, we introduce a certain kind of labelled rooted binary trees. 
\begin{Def}\label{diskorder}
A  rooted  binary tree is called   \emph{di-sk tree} if its nodes are labelled either with  $\oplus$ or $\ominus$ and no node has the same label as its right child.  We use  the \emph{in-order} (tranversal) to compare nodes on di-sk trees:  starting with the root node, we recursively traverse the left subtree to parent then to the right subtree if any. See Fig.~\ref{di-sk} for a di-sk tree and its in-order.
The number of nodes labelled by $\ominus$ (resp.~$\oplus$)  in a di-sk tree $T$  is denoted by $n_\ominus(T)$ (resp.~$n_\oplus(T)$). The set of all di-sk trees with $n-1$ nodes is denoted as $\D\T_n$. 
\end{Def}

 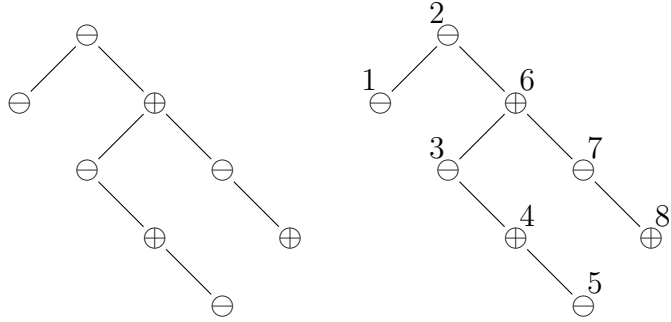
\begin{figure}
\begin{tikzpicture}[scale=0.3]
\draw[-] (17,10) to (19,12);
\draw[-] (20,12) to (22,10);
\draw[-] (22,9) to (20,7);
\draw[-] (23,9) to (25,7);
\draw[-] (20,6) to (22,4);
\draw[-] (26,6) to (28,4);
\draw[-] (23,3) to (25,1);

\draw[-] (33,10) to (35,12);
\draw[-] (36,12) to (38,10);
\draw[-] (38,9) to (36,7);
\draw[-] (39,9) to (41,7);
\draw[-] (36,6) to (38,4);
\draw[-] (42,6) to (44,4);
\draw[-] (39,3) to (41,1);

\node at (16.5,9.5) {$\ominus$};
\node at (19.5,12.5) {$\ominus$};
\node at (22.5,9.5) {$\oplus$};
\node at (19.5,6.5) {$\ominus$};
\node at (25.5,6.5) {$\ominus$};
\node at (22.5,3.5) {$\oplus$};
\node at (25.5,0.5) {$\ominus$};
\node at (28.5,3.5) {$\oplus$};

\node at (32.5,9.5) {$\ominus$};
\node at (35.5,12.5) {$\ominus$};
\node at (38.5,9.5) {$\oplus$};
\node at (35.5,6.5) {$\ominus$};
\node at (41.5,6.5) {$\ominus$};
\node at (38.5,3.5) {$\oplus$};
\node at (41.5,0.5) {$\ominus$};
\node at (44.5,3.5) {$\oplus$};
%in-order for nodes
\node at (32,10.5) {$1$};
\node at (35,13.5) {$2$};
\node at (39,10.5) {$6$};
\node at (35,7.5) {$3$};
\node at (42,7.5) {$7$};
\node at (39,4.5) {$4$};
\node at (42,1.5) {$5$};
\node at (45,4.5) {$8$};
\end{tikzpicture}
\caption{A di-sk tree and its in-order.\label{di-sk}}
\end{figure}

The following bijection between separable permutations and  di-sk trees is essentially due to Shapiro and Stephens~\cite{sh}.

\begin{theorem}\label{disktr}
There exists a  bijection 
 $\Phi: \S_n(2413,3142)\rightarrow\D\T_n$ such that 
 \begin{align}\label{minu:des}
 i\in \DES(\pi) \quad\Leftrightarrow \quad \text{the $i$th node (by in-order) of $\Phi(\pi)$ is $\ominus$},
 \end{align}
 where $\DES(\pi)$ is the set of all descents of $\pi$. Consequently, 
 \begin{equation}\label{tree:version}
S_n(t)=\sum_{T\in \D\T_n} t^{n_\ominus(T)}.
\end{equation}
  \end{theorem}
  
 \begin{proof} The bijection $\Phi$ can be constructed recursively. Set $\Phi(12)=\oplus$ and $\Phi(21)=\ominus$ (by convention $\Phi(1)=\emptyset$). For $\pi\in\S_n(2413,3142)$ with $n\geq3$, we find the greatest index $i\in[n-1]$ such that either 
 $$\min\{\pi_1,\ldots,\pi_i\}>\max\{\pi_{i+1},\ldots,\pi_n\}\quad\text{or}\quad\max\{\pi_1,\ldots,\pi_i\}<\min\{\pi_{i+1},\ldots,\pi_n\}.$$
 By Proposition~\ref{desides}, such an index $i$ exists and is unique. If the first inequality holds, then $\pi=\omega\ominus\upsilon$ with $\omega=(\pi_1+i-n)\cdots(\pi_i+i-n)\in\S_i(2413,3142)$ and $\upsilon=\pi_{i+1}\cdots\pi_{n}\in\S_{n-i}(2413,3142)$. Define $\Phi(\pi)=(\Phi(\omega),\ominus,\Phi(\upsilon))$, the tree with the left subtree $\Phi(\omega)$ and the right subtree $\Phi(\upsilon)$ attached to the root $\ominus$. Otherwise, we have $\pi=\omega\oplus\upsilon$, where $\omega=\pi_1\cdots\pi_i\in\S_i(2413,3142)$ and $\upsilon=(\pi_{i+1}-i)\cdots(\pi_{n}-i)\in\S_{n-i}(2413,3142)$. We then define $\Phi(\pi)=(\Phi(\omega),\oplus,\Phi(\upsilon))$, the tree with the left subtree $\Phi(\omega)$ and the right subtree $\Phi(\upsilon)$ attached to the root $\oplus$.
 
 Since $i$ is chosen to be the greatest,  the root of $\Phi(\pi)$ has a different label than its right child, which is the root of $\Phi(\upsilon)$. This shows that $\Phi(\pi)$ is a di-sk tree and  so $\Phi$ is well defined. For example, if $\pi=984132756$, then $\Phi(\pi)$ is the di-sk tree in Fig.~\ref{di-sk}. 
 It is not hard to check that $\Phi$ is a bijection satisfying the property~\eqref{minu:des}, which completes the proof.
 \end{proof}

\begin{corollary}
For $n\geq2$, we have the following recurrence relation for $S_n(t)$:
\begin{equation}\label{rec:ss}
S_n(t)=(1+t)S_{n-1}(t)+t\sum_{j=1}^{n-2}S_j(t)\biggl(S_{n-j-1}(t)+\sum_{i=1}^{n-j-1}S_i(t)S_{n-j-i}(t)\biggr).
\end{equation}
Equivalently, 
\begin{equation*}\label{cubic}
S(t,z)=z+(1+t)zS(t,z)+tzS^2(t,z)+tS^3(t,z),
\end{equation*}
where $S(t,z):=\sum_{n\geq1}S_n(t)z^n$.
\end{corollary}
\begin{proof}

For $n\geq2$, let 
$$
S_n^{(1)}(t):=\sum_{T\in\D\T^{\oplus}_n}t^{n_\ominus(T)}\qquad\text{and}\qquad S_n^{(2)}(t):=\sum_{T\in\D\T^{\ominus}_n}t^{n_\ominus(T)},
$$
where $\D\T^{\oplus}_n$ and  $\D\T^{\ominus}_n$ are the set of all di-sk trees in $\D\T_n$ with root labelled by $\oplus$ and $\ominus$, respectively. 
It follows from~\eqref{tree:version} that
\begin{equation}\label{trrr}
S_n(t)=S_n^{(1)}(t)+S_n^{(2)}(t)
\end{equation}
 if $n\geq2$. For convenience, we set $S_1^{(1)}(t)=S_1^{(2)}(t)=1$. 
We claim that for $n\geq2$:
\begin{equation}\label{s1}
S_n^{(1)}(t)=\sum_{j=1}^{n-1}S_j(t)S_{n-j}^{(2)}(t)\qquad\text{and}\qquad S_n^{(2)}(t)=t\sum_{j=1}^{n-1}S_j(t)S_{n-j}^{(1)}(t).
\end{equation}
Actually,  any di-sk tree in $\D\T^{\oplus}_n$ can be constructed from a root labelled $\oplus$ by attaching a di-sk tree on the left branch and a di-sk tree with a root labelled $\ominus$ on the right branch. This gives the first expression in~\eqref{s1}. The second expression in~\eqref{s1} follows by similar decomposition of a di-sk tree in $\D\T^{\ominus}_n$.

For $n\geq1$, let 
$$
S^*_n(t):=tS_n^{(1)}(t)+S_n^{(2)}(t).
$$
Note that $S^*_1(t)=1+t$.
It follows from~\eqref{s1} and~\eqref{trrr} that, for $n\geq2$
\begin{equation}\label{rec:s}
S_n(t)=\sum_{j=1}^{n-1}S_j(t)(S_{n-j}^{(2)}(t)+tS_{n-j}^{(1)}(t))=\sum_{j=1}^{n-1}S_j(t)S^*_{n-j}(t)
\end{equation}
and 
\begin{align*}
S^*_n(t)=\sum_{j=1}^{n-1}S_j(t)(tS_{n-j}^{(2)}(t)+tS_{n-j}^{(1)}(t))=tS_{n-1}(t)+t\sum_{j=1}^{n-1}S_j(t)S_{n-j}(t).
\end{align*}
Substituting the latter into~\eqref{rec:s}, we get~\eqref{rec:ss}.
\end{proof}

\begin{remark}
In his 2008 thesis~\cite[Example~1.6.7]{dra}, Drake uses an inversion  theorem for labelled trees to compute the generating function for di-sk trees by the number of $\ominus$-nodes. However, there is a mistake in his computation, which erroneously leads him to a different sequence in OEIS~\href{https://oeis.org/A089447}{\tt{oeis:A089447}}.
\end{remark}

It is not hard to show that if $A(t)$ and $B(t)$ are $\gamma$-positive of center $m/2$ and $n/2$ respectively, then $A(t)B(t)$ is $\gamma$-positive of center 
 $(m+n)/2$.
The $\gamma$-positivity of $S_n(t)$ then follows from~\eqref{rec:ss} by induction on $n$. 
Let $\Gamma_n(x):=\sum_{k=0}^{\lfloor\frac{n-1}{2}\rfloor}\gamma_{n,k}^Sx^k$ be the $\gamma$-polynomial of $S_n(t)$, where $\gamma_{n,k}^S$ is defined by~\eqref{gam:scho}. In fact, the recurrence relation~\eqref{rec:ss} for $S_n(t)$ is equivalent to the following recurrence for $\Gamma_n(x)$, because 
$$
S_n(t)=(1+t)^{n-1}\Gamma_n(x),\quad\text{with $x=\frac{t}{(1+t)^2}$}.
$$
\begin{corollary}The recurrence relation for $\Gamma_n(x)$ is
$$
\Gamma_n(x)=\Gamma_{n-1}(x)+x\sum_{j=1}^{n-2}\Gamma_j(x)\biggl(\Gamma_{n-j-1}(x)+\sum_{i=1}^{n-j-1}\Gamma_i(x)\Gamma_{n-j-i}(x)\biggr)
$$
with initial value $\Gamma_1(x)=1$.
\end{corollary}

%%%%%%%%%%%%%%%%%%%%%%%%%%%%%%%%%
\section{Proof of Theorem~\ref{gamma:s}}\label{sec:gammasep}
%%%%%%%%%%%%%%%%%%%%%%%%%%%%%%%%%

This section is devoted to a purely combinatorial  proof of Theorem~\ref{gamma:s}. We first introduce a natural action on di-sk trees, which gives an interpretation  for $\gamma^S_{n,k}$.

\subsection{A natural action on di-sk trees}
\begin{Def}
Given a di-sk tree, its {\em right chain} (or simply {\em chain}) is any maximal chain composed of only right edges. For instance, there are $3$ right chains of the di-sk trees in Fig.~\ref{di-sk}, which are the chains $1$, $2-6-7-8$ and $3-4-5$. The {\em length of a  chain} $C$, denoted $|C|$, is the number of nodes of $C$. We distinguish a chain by {\em even} or {\em odd} according to its length. The first node and the last node of a chain are called the {\em head} and the {\em tail} of this chain, respectively. A chain with its head labelled by $\ominus$ (resp.~$\oplus$) is called a {\em$\ominus$-chain} (resp.~{\em$\oplus$-chain}).
\end{Def}

Since no node has the same label as its right child in a di-sk tree, the labels on each chain must alternate, and therefore is completely decided by the label of its head. For each $x\in[n-1]$ and $T\in\D\T_n$, we define the action $\varphi_x$ as: if the $x$-th node of $T$ is the head of an odd chain $C$ of $T$, then   $\varphi_x(T)$ is obtained from $T$ by changing all labels of nodes on $C$; otherwise, $\varphi_x(T)=T$. Clearly, $\varphi_x$ is an involution acting on $\D\T_n$ and $\varphi_x$ and $\varphi_y$ commute for all $x,y\in[n-1]$. Thus, for any subset $S\subseteq[n-1]$ we can define the function $\varphi_S:\D\T_n\rightarrow\D\T_n$ by $\varphi_S=\prod_{x\in S}\varphi_x$.  Hence the group $\Z_2^{n-1}$ acts on $\D\T_n$ via the function $\varphi_S$, $S\subseteq[n-1]$. 

Let us introduce the subset 
$$
\D\T_{n,k}^{1}:=\{T\in\D\T_n :  \text{ all odd chains are $\oplus$-chains and }n_\ominus(T)=k\},
$$
For each $T\in\D\T_{n,k}^{1}$, since all its odd chains are $\oplus$-chains, we see the total number of odd chains is given by
$$n_\oplus(T)-n_\ominus(T)=n_\oplus(T)+n_\ominus(T)-2n_\ominus(T)=n-1-2k.$$ 
Let $\Orb(T)$ be the orbit of $T$ under the above $\Z_2^{n-1}$-action. Note that $T$ is the only di-sk tree in $\Orb(T)\cap\D\T_{n,k}^{1}$ with the minimal number of $\ominus$-nodes. It then follows that 
$$
\sum_{S\in\Orb(T)}t^{n_\ominus(S)}=t^k(1+t)^{n-1-2k},
$$
since each $S\in\Orb(T)$ can be obtained from $T$ by turning some of the  $n-1-2k$  odd $\oplus$-chains into odd $\ominus$-chains. Summing over all the orbits we get

\begin{lemma}\label{int:1}
We have the following $\gamma$-expansion
$$
S_n(t)=\sum_{T\in \D\T_n} t^{n_\ominus(T)}=\sum_{k\geq 0}^{\lfloor(n-1)/2\rfloor}|\D\T_{n,k}^1| t^k (1+t)^{n-1-2k}.
$$
\end{lemma}

By property~\eqref{minu:des}, the bijection $\Phi$ in Theorem~\ref{disktr}  induces a bijection between the following two subsets:
\begin{align*}
\S_{n,k}^S:= & \{\pi \in \S_n(2413,3142) : \dd(\pi)=0, \des(\pi)=k\},\\
\D\T_{n,k}^{2}:= &\{T\in\D\T_n : \text{ $T$ has no consecutive $\ominus$, its first node is $\oplus$ and}\,\,\, n_\ominus(T)=k\}.
\end{align*}
Therefore, in view of Lemma~\ref{int:1}, to show that $\gamma^S_{n,k}=|\S_{n,k}^S|$, it  suffices to build a bijection between $\D\T_{n,k}^{1}$ and $\D\T_{n,k}^{2}$.

%%%%%%%%%%%%%%%%%%%%%%%%%%%%%%%%%%%%%
\subsection{A bijection between $\D\T_{n,k}^{2}$ and $\D\T_{n,k}^{1}$}%%
%%%%%%%%%%%%%%%%%%%%%%%%%%%%%%%%%%%%%
 
\begin{theorem}\label{gammabij}
There exists a bijection $\Psi:\D\T_{n,k}^{2}\rightarrow\D\T_{n,k}^{1}$ for each $0\leq k\leq\lfloor(n-1)/2\rfloor$. Consequently, Theorem~\ref{gamma:s} is true.
\end{theorem}

%The rest of this section is devoted to the construction of $\Psi$.
\begin{Def}
We say two nodes are {\em at the same level} if they are connected by a sequence of left edges. Then, whether two chains are at the same level or not is according to their heads. For example, in Fig.~\ref{di-sk},  chains $1$ and $2-6-7-8$ are at the same level, which is different from that of chain $3-4-5$.  We also order the chains by the in-order of their heads and denote $A>B$ if chain $A$ appears before  chain $B$. 
\end{Def}

Our $\Psi$ when restricted to $\D\T^{1}_{n,k}\cap \D\T^{2}_{n,k}$ is simply identity. So, we only need to define the mapping $\Psi$ from  $\D\T^{2}_{n,k}\backslash \D\T^{1}_{n,k}$ to $\D\T^{1}_{n,k}\backslash \D\T^{2}_{n,k}$.

%%%%%%%%%%%%%%%%%%%%%%%%%%%%%%%%%%%%%%%%%%%%%%%%
\subsubsection{The construction of $\Psi:\D\T^{2}_{n,k}\backslash \D\T^{1}_{n,k}\rightarrow\D\T^{1}_{n,k}\backslash \D\T^{2}_{n,k}$}%%%%%%%%%%%%%%%%%%%%%%%%%%%%%%
%%%%%%%%%%%%%%%%%%%%%%%%%%%%%%%%%%%%%%%%%%%%%%%%%

Our goal is to remove all the odd $\ominus$-chains, and create consecutive $\ominus$ nodes and/or a new first node labelled $\ominus$. The construction of $\Psi$ will be a sequence of ``cut-and-paste'' in nature. We first define the ``adjoint'' of a given odd $\ominus$-chain. This is crucial in finding where to apply our ``cut-and-paste'' operation.

\begin{Def}[Adjoints for odd $\ominus$-chains]\label{adjoint}
Given a di-sk tree $T\in \D\T^{2}_{n,k}\backslash \D\T^{1}_{n,k}$, for each odd $\ominus$-chain $C$ of $T$, we find a unique odd $\oplus$-chain $C^{*}$ of $T$, called the {\em adjoint of $C$}, at the same level, according to the two cases below. We denote $F$ (resp.~$L$) the first (resp.~the last) node at the same level as the head of $C$.
\begin{itemize}
\item[Case I:] $L$ is the root or a non-root labelled $\oplus$. This forces $F$ to be labelled $\oplus$, either because it is the first node or because the left parent of $L$ is labelled $\ominus$ and we forbid consecutive $\ominus$. We scan all the nodes at the same level, between $F$ and the head of $C$, find the closest one to $C$ that is labelled $\oplus$ (at least we have $F$), it must be leading an odd chain, denote this chain as $C^{*}$. Consequently, all chains between $C$ and $C^{*}$ are even $\ominus$-chains. 
\item[Case II:] $L$ is a non-root labelled $\ominus$. We scan all the nodes at the same level, between $L$ and the head of $C$, find the closest one to $C$ that is labelled $\ominus$ (at least we have $L$). Consequently, all chains between $C$ and this closest $\ominus$-chain are $\oplus$-chains and we denote the last one of them as $C^{*}$. So $C^{*}$ must be an odd $\oplus$-chain. 
\end{itemize}
\end{Def}

For comparing with our construction of $\Psi^{-1}$, we split the above two cases into  $6$ cases as depicted in Fig.~\ref{findadj}. Note that the dash line means this portion of the tree could be of any type, including the empty set case.

 %%%%%%%%%%%%forward six cases%%%%%%%%%%%%
\begin{figure}
\begin{tikzpicture}[scale=0.21]
%plot the trees
%Case I
\draw[-,dashed] (0,0) to (2,2);
\draw[-] (3,3) to (5,5);
\draw[-] (6,6) to (8,8);
\draw[-] (9,9) to (11,11);
\draw[-] (12,12) to (14,14);
\draw[-,dashed] (15,15) to (17,17);
\draw[-,dashed] (15,14) to (17,12);
\draw[-] (6,5) to (8,3);
\draw[-] (9,2) to (11,0);
\draw[-] (12,-1) to (14,-3);
\draw[-] (9,8) to (11,6);
\draw[-] (12,11) to (14,9);
\draw[-] (15,8) to (17,6);
%Case II
\draw[-,dashed] (0,-25) to (2,-23);
\draw[-] (3,-22) to (5,-20);
\draw[-] (6,-19) to (8,-17);
\draw[-] (9,-16) to (11,-14);
\draw[-] (12,-13) to (14,-11);
\draw[-,dashed] (15,-10) to (17,-8);
\draw[-,dashed] (15,-11) to (17,-13);
\draw[-] (3,-23) to (5,-25);
\draw[-] (6,-26) to (8,-28);
\draw[-] (6,-20) to (8,-22);
\draw[-] (9,-23) to (11,-25);
\draw[-] (12,-26) to (14,-28);
\draw[-] (9,-17) to (11,-19);
\draw[-] (12,-14) to (14,-16);
\draw[-] (15,-17) to (17,-19);
%Case III
\draw[-,dashed] (20,-3) to (22,-1);
\draw[-] (23,0) to (25,2);
\draw[-] (26,3) to (28,5);
\draw[-] (29,6) to (31,8);
\draw[-,dashed] (32,9) to (34,11);
\draw[-] (26,2) to (28,0);
\draw[-,dashed] (32,8) to (34,6);
\draw[-,dashed] (35,11) to (37,9);
\draw[-] (34,12) to (32,14);
\draw[-,dashed] (31,15) to (29,17);
%Case IV
\draw[-,dashed] (20,-28) to (22,-26);
\draw[-] (23,-25) to (25,-23);
\draw[-] (26,-22) to (28,-20);
\draw[-] (29,-19) to (31,-17);
\draw[-,dashed] (32,-16) to (34,-14);
\draw[-] (23,-26) to (25,-28);
\draw[-] (26,-29) to (28,-31);
\draw[-] (26,-23) to (28,-25);
\draw[-,dashed] (32,-17) to (34,-19);
\draw[-,dashed] (35,-14) to (37,-16);
\draw[-] (34,-13) to (32,-11);
\draw[-,dashed] (31,-10) to (29,-8);
%Case V
\draw[-,dashed] (41,-3) to (43,-1);
\draw[-] (44,0) to (46,2);
\draw[-,dashed] (47,3) to (49,5);
\draw[-,dashed] (50,6) to (52,8);
\draw[-] (53,9) to (55,11);
\draw[-,dashed] (56,12) to (58,14);
\draw[-,dashed] (50,5) to (52,3);
\draw[-] (53,8) to (55,6);
\draw[-] (56,5) to (58,3);
\draw[-,dashed] (56,11) to (58,9);
\draw[-,dashed] (59,14) to (61,12);
\draw[-] (58,15) to (56,17);
\draw[-,dashed] (55,18) to (53,20);
%Case VI
\draw[-,dashed] (41,-31) to (43,-29);
\draw[-] (44,-28) to (46,-26);
\draw[-,dashed] (47,-25) to (49,-23);
\draw[-,dashed] (50,-22) to (52,-20);
\draw[-] (53,-19) to (55,-17);
\draw[-,dashed] (56,-16) to (58,-14);
\draw[-] (47,-26) to (49,-28);
\draw[-] (50,-29) to (52,-31);
\draw[-,dashed] (50,-23) to (52,-25);
\draw[-] (53,-20) to (55,-22);
\draw[-] (56,-23) to (58,-25);
\draw[-,dashed] (56,-17) to (58,-19);
\draw[-,dashed] (59,-14) to (61,-16);
\draw[-] (58,-13) to (56,-11);
\draw[-,dashed] (55,-10) to (53,-8);

%put nodes
%Case I-1
\node at (2.5,2.5) {$\oplus$};
\node at (5.5,5.5) {$\ominus$};
\node at (8.5,8.5) {$\ominus$};
\node at (11.5,11.5) {$\ominus$};
\node at (14.5,14.5) {$\oplus$};
\node at (8.5,2.5) {$\oplus$};
\node at (11.5,5.5) {$\oplus$};
\node at (14.5,8.5) {$\oplus$};
\node at (11.5,-0.5) {$\ominus$};
\node at (14.5,-3.5) {$\oplus$};
\node at (17.5,5.5) {$\ominus$};
%Case I-2
\node at (2.5,-22.5) {$\oplus$};
\node at (5.5,-19.5) {$\ominus$};
\node at (8.5,-16.5) {$\ominus$};
\node at (11.5,-13.5) {$\ominus$};
\node at (14.5,-10.5) {$\oplus$};
\node at (5.5,-25.5) {$\ominus$};
\node at (8.5,-22.5) {$\oplus$};
\node at (11.5,-19.5) {$\oplus$};
\node at (14.5,-16.5) {$\oplus$};
\node at (8.5,-28.5) {$\oplus$};
\node at (11.5,-25.5) {$\ominus$};
\node at (14.5,-28.5) {$\oplus$};
\node at (17.5,-19.5) {$\ominus$};
%Case I-3
\node at (22.5,-0.5) {$\oplus$};
\node at (25.5,2.5) {$\ominus$};
\node at (28.5,5.5) {$\ominus$};
\node at (31.5,8.5) {$\oplus$};
\node at (34.5,11.5) {$\oplus$};
\node at (31.5,14.5) {$\ominus$};
\node at (28.5,-0.5) {$\oplus$};
%Case I-4
\node at (22.5,-25.5) {$\oplus$};
\node at (25.5,-22.5) {$\ominus$};
\node at (28.5,-19.5) {$\ominus$};
\node at (31.5,-16.5) {$\oplus$};
\node at (34.5,-13.5) {$\oplus$};
\node at (31.5,-10.5) {$\ominus$};
\node at (25.5,-28.5) {$\ominus$};
\node at (28.5,-31.5) {$\oplus$};
\node at (28.5,-25.5) {$\oplus$};
%Case II-1
\node at (43.5,-0.5) {$\oplus$};
\node at (46.5,2.5) {$\ominus$};
\node at (49.5,5.5) {$\oplus$};
\node at (52.5,8.5) {$\oplus$};
\node at (55.5,11.5) {$\ominus$};
\node at (58.5,14.5) {$\ominus$};
\node at (55.5,17.5) {$\oplus$};
\node at (55.5,5.5) {$\ominus$};
\node at (58.5,2.5) {$\oplus$};
%Case II-2
\node at (43.5,-28.5) {$\oplus$};
\node at (46.5,-25.5) {$\ominus$};
\node at (49.5,-22.5) {$\oplus$};
\node at (52.5,-19.5) {$\oplus$};
\node at (55.5,-16.5) {$\ominus$};
\node at (58.5,-13.5) {$\ominus$};
\node at (55.5,-10.5) {$\oplus$};
\node at (49.5,-28.5) {$\oplus$};
\node at (52.5,-31.5) {$\ominus$};
\node at (55.5,-22.5) {$\ominus$};
\node at (58.5,-25.5) {$\oplus$};
%label
%caption
\node at (8,-6) {Case I-1};
\node at (8,-35) {Case I-2};
\node at (28,-6) {Case I-3};
\node at (28,-35) {Case I-4};
\node at (51,-6) {Case II-1};
\node at (51,-35) {Case II-2};
%L
\node at (35,13) {$L$};
\node at (35,-12) {$L$};
\node at (59,16) {$L$};
\node at (59,-12) {$L$};
%Case I
\node at (2,4) {$C^*$};
%\node at (5,7) {$C_2$};
%\node at (8,10) {$C_1$};
\node at (11,13) {$C$};
%Case II
\node at (2,-21) {$C^*$};
%\node at (5,-18) {$C_2$};
%\node at (8,-15) {$C_1$};
\node at (11,-12) {$C$};
%Case III
\node at (22,1) {$C^*$};
%\node at (25,4) {$C_1$};
\node at (28,7) {$C$};
%Case IV
\node at (22,-24) {$C^*$};
%\node at (25,-21) {$C_1$};
\node at (28,-18) {$C$};
%Case V
\node at (46,4) {$C$};
\node at (52,10) {$C^*$};
%\node at (55,13) {$N_*$};
%Case VI
\node at (46,-24) {$C$};
\node at (52,-18) {$C^*$};
%\node at (55,-15) {$N_*$};
\end{tikzpicture}
\caption{{\bf I-1.} $L$ is the root and $|C^*|=1$; {\bf I-2.} $L$ is the root and $|C^*|>1$; {\bf I-3.} $L$ is a non-root labelled $\oplus$  and $|C^*|=1$; {\bf I-4.} $L$ is a non-root labelled $\oplus$  and $|C^*|>1$; {\bf II-1.} $L$ is a non-root labelled  $\ominus$  and $|C|=1$; {\bf II-2.} $L$ is a non-root labelled $\ominus$  and $|C|>1$.}
\label{findadj}
\end{figure}
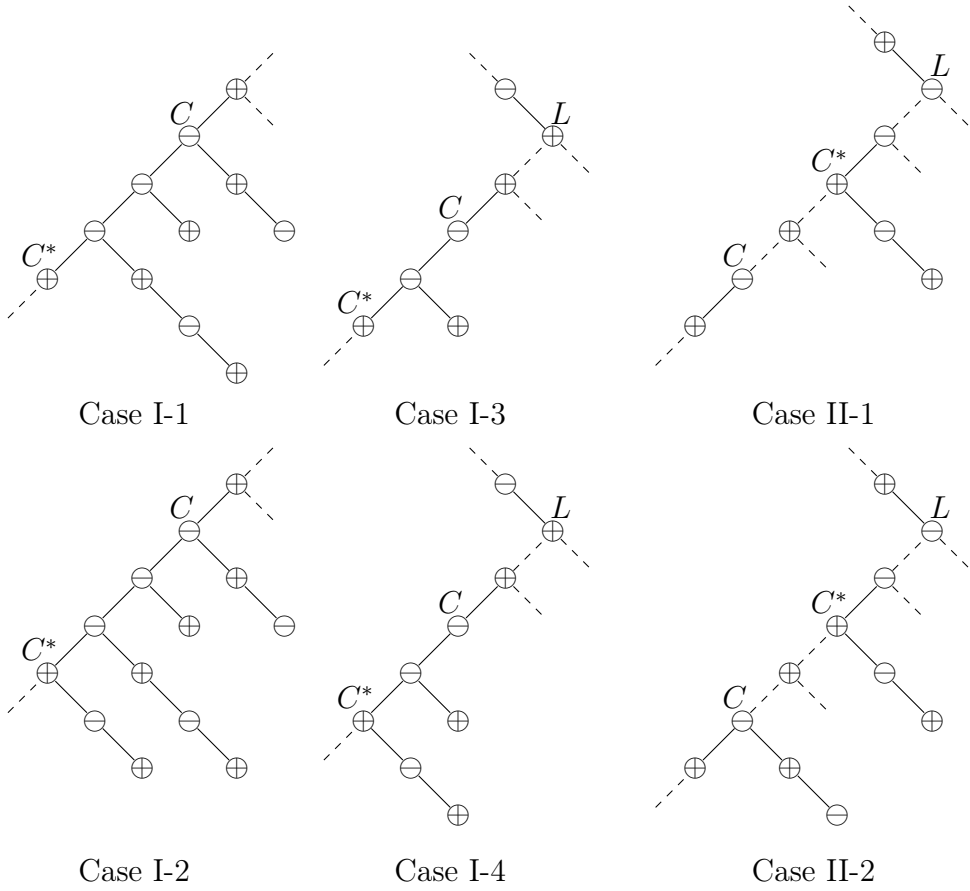

%For two chains $A$ and $B$, we write $A>B$ if  the head of $A$ is visited before the head of $B$ by the in-order. 
Now we are ready to define $\Psi:\D\T^{2}_{n,k}\backslash \D\T^{1}_{n,k}\rightarrow\D\T^{1}_{n,k}\backslash \D\T^{2}_{n,k}$. Take any $T\in \D\T^{2}_{n,k}\backslash \D\T^{1}_{n,k}$. Find all the adjoints $C^*$ for all the odd $\ominus$-chains $C$. We perform the following operations for all pairs $(C,C^*)$ and denote the resulting tree as $\Psi(T)$.

\begin{itemize}
\item[Case I:] $C^*>C$. We cut off the tail of $C^*$ (labelled $\oplus$), together with its left subtree if any, and attach it to the tail of $C$ (labelled $\ominus$) from right.
\item[Case II:] $C>C^*$. We cut off the tail of $C$ (labelled $\ominus$), together with its left subtree if any, and attach it to the tail of $C^*$ (labelled $\oplus$) from right.
\end{itemize}

\begin{remark}
Note that in both cases, no two pairs $(C_1,C_1^{*})$ and $(C_2,C_2^{*})$ can intertwine with each other. By applying $\Psi$, we eliminate an odd $\ominus$-chain, and create either a consecutive pair of $\ominus$ nodes, or a new first node labelled $\ominus$, so $\Psi(T)$ is indeed in $\D\T^{1}_{n,k}\backslash \D\T^{2}_{n,k}$ and $\Psi$ is well defined. 
\end{remark}

We give in Fig.~\ref{findL} the corresponding $\Psi(T)$ for the six cases included in Fig.~\ref{findadj}. In the following, we show that $\Psi$ is a bijection by constructing its inverse. 

%%%%%%%%%%%%%%%%%%%%%%%%%%%%%%%%%%%%%%%%%%%%%%%%%
\subsubsection{The construction of $\Psi^{-1}:\D\T^{1}_{n,k}\backslash \D\T^{2}_{n,k}\rightarrow\D\T^{2}_{n,k}\backslash \D\T^{1}_{n,k}$}%%%%%%%%%%%%%%%%%%%%%%%%%%
%%%%%%%%%%%%%%%%%%%%%%%%%%%%%%%%%%%%%%%%%%%%%%%%%

\begin{lemma}\label{head}
For a pair of consecutive $\ominus$-nodes in a di-sk $T$ in $\D\T^{1}_{n,k}\backslash \D\T^{2}_{n,k}$, exactly one of $\ominus$ in such a pair is the head of certain chain.
\end{lemma}
\begin{proof}
Since all $\ominus$-chains of $T$ are even, any two consecutive $\ominus$-nodes can not both be the heads of certain chains. The result then follows from the fact that two $\ominus$-nodes, neither of which is the head of certain chain, can not form a pair of consecutive  $\ominus$-nodes.
\end{proof}

Because of Lemma~\ref{head}, an even $\ominus$-chain of  $T\in\D\T^{1}_{n,k}\backslash \D\T^{2}_{n,k}$, whose head either forms a pair of consecutive $\ominus$-nodes (with another adjacency $\ominus$-node) or is  the first node of $T$,  is called a {\em bad even $\ominus$-chain}.

%%%%%backward six cases%%%%%%%%%%%%
\begin{figure}
\begin{tikzpicture}[scale=0.21]
%plot the trees
%Case I
\draw[-] (6,6) to (8,8);
\draw[-] (9,9) to (11,11);
\draw[-] (12,12) to (14,14);
\draw[-,dashed] (15,15) to (17,17);
\draw[-,dashed] (15,14) to (17,12);
\draw[-] (6,5) to (8,3);
\draw[-] (9,2) to (11,0);
\draw[-] (12,-1) to (14,-3);
\draw[-] (9,8) to (11,6);
\draw[-] (12,11) to (14,9);
\draw[-] (15,8) to (17,6);
\draw[-] (18,5) to (20,3);
\draw[-,dashed] (20,2) to (18,0);
%Case II
\draw[-,dashed] (0,-25) to (2,-23);
\draw[-] (3,-22) to (5,-20);
\draw[-] (6,-19) to (8,-17);
\draw[-] (9,-16) to (11,-14);
\draw[-] (12,-13) to (14,-11);
\draw[-,dashed] (15,-10) to (17,-8);
\draw[-,dashed] (15,-11) to (17,-13);
\draw[-] (3,-23) to (5,-25);
\draw[-] (6,-20) to (8,-22);
\draw[-] (9,-23) to (11,-25);
\draw[-] (12,-26) to (14,-28);
\draw[-] (9,-17) to (11,-19);
\draw[-] (12,-14) to (14,-16);
\draw[-] (15,-17) to (17,-19);
\draw[-] (18,-20) to (20,-22);
%Case III
\draw[-] (28,3) to (30,5);
\draw[-] (31,6) to (33,8);
\draw[-,dashed] (34,9) to (36,11);
\draw[-] (28,2) to (30,0);
\draw[-] (31,5) to (33,3);
\draw[-,dashed] (33,2) to (32,1);
\draw[-,dashed] (34,8) to (36,6);
\draw[-,dashed] (37,11) to (39,9);
\draw[-] (36,12) to (34,14);
\draw[-,dashed] (33,15) to (31,17);
%Case IV
\draw[-,dashed] (22,-28) to (24,-26);
\draw[-] (25,-25) to (27,-23);
\draw[-] (28,-22) to (30,-20);
\draw[-] (31,-19) to (33,-17);
\draw[-,dashed] (34,-16) to (36,-14);
\draw[-] (25,-26) to (27,-28);
\draw[-] (28,-23) to (30,-25);
\draw[-] (31,-20) to (33,-22);
\draw[-,dashed] (34,-17) to (36,-19);
\draw[-,dashed] (37,-14) to (39,-16);
\draw[-] (36,-13) to (34,-11);
\draw[-,dashed] (33,-10) to (31,-8);
%Case V
%\draw[-,dashed] (47,4) to (49,5);
\draw[-,dashed] (50,6) to (52,8);
\draw[-] (53,9) to (55,11);
\draw[-,dashed] (56,12) to (58,14);
\draw[-,dashed] (50,5) to (52,3);
\draw[-] (53,8) to (55,6);
\draw[-] (56,5) to (58,3);
\draw[-] (59,2) to (61,0);
\draw[-,dashed] (56,11) to (58,9);
\draw[-,dashed] (59,14) to (61,12);
\draw[-] (58,15) to (56,17);
\draw[-,dashed] (55,18) to (53,20);
\draw[-] (61,-1) to (59,-3);
\draw[-,dashed] (58,-4) to (56,-6);
%Case VI
\draw[-,dashed] (44,-28) to (46,-26);
\draw[-,dashed] (47,-25) to (49,-23);
\draw[-,dashed] (50,-22) to (52,-20);
\draw[-] (53,-19) to (55,-17);
\draw[-,dashed] (56,-16) to (58,-14);
\draw[-] (47,-26) to (49,-28);
\draw[-,dashed] (50,-23) to (52,-25);
\draw[-] (53,-20) to (55,-22);
\draw[-] (56,-23) to (58,-25);
\draw[-] (59,-26) to (61,-28);
\draw[-,dashed] (56,-17) to (58,-19);
\draw[-,dashed] (59,-14) to (61,-16);
\draw[-] (58,-13) to (56,-11);
\draw[-,dashed] (55,-10) to (53,-8);

%put nodes
%Case I
\node at (5.5,5.5) {\red{$\ominus$}};
\node at (8.5,8.5) {$\ominus$};
\node at (11.5,11.5) {$\ominus$};
\node at (14.5,14.5) {$\oplus$};
\node at (8.5,2.5) {$\oplus$};
\node at (11.5,5.5) {$\oplus$};
\node at (14.5,8.5) {$\oplus$};
\node at (11.5,-0.5) {$\ominus$};
\node at (14.5,-3.5) {$\oplus$};
\node at (17.5,5.5) {$\ominus$};
\node at (20.5,2.5) {$\oplus$};
%Case II
\node at (2.5,-22.5) {$\oplus$};
\node at (5.5,-19.5) {\red{$\ominus$}};
\node at (8.5,-16.5) {$\ominus$};
\node at (11.5,-13.5) {$\ominus$};
\node at (14.5,-10.5) {$\oplus$};
\node at (5.5,-25.5) {\red{$\ominus$}};
\node at (8.5,-22.5) {$\oplus$};
\node at (11.5,-19.5) {$\oplus$};
\node at (14.5,-16.5) {$\oplus$};
\node at (11.5,-25.5) {$\ominus$};
\node at (14.5,-28.5) {$\oplus$};
\node at (17.5,-19.5) {$\ominus$};
\node at (20.5,-22.5) {$\oplus$};
%Case III
\node at (27.5,2.5) {\red{$\ominus$}};
\node at (30.5,5.5) {$\ominus$};
\node at (33.5,8.5) {$\oplus$};
\node at (36.5,11.5) {$\oplus$};
\node at (33.5,14.5) {\red{$\ominus$}};
\node at (30.5,-0.5) {$\oplus$};
\node at (33.5,2.5) {$\oplus$};
%Case IV
\node at (24.5,-25.5) {$\oplus$};
\node at (27.5,-22.5) {\red{$\ominus$}};
\node at (30.5,-19.5) {$\ominus$};
\node at (33.5,-16.5) {$\oplus$};
\node at (36.5,-13.5) {$\oplus$};
\node at (33.5,-10.5) {$\ominus$};
\node at (27.5,-28.5) {\red{$\ominus$}};
\node at (30.5,-25.5) {$\oplus$};
\node at (33.5,-22.5) {$\oplus$};
%Case V
\node at (49.5,5.5) {$\oplus$};
\node at (52.5,8.5) {$\oplus$};
\node at (55.5,11.5) {\red{$\ominus$}};
\node at (58.5,14.5) {$\ominus$};
\node at (55.5,17.5) {$\oplus$};
\node at (55.5,5.5) {$\ominus$};
\node at (58.5,2.5) {$\oplus$};
\node at (61.5,-0.5) {\red{$\ominus$}};
\node at (58.5,-3.5) {$\oplus$};
%Case VI
\node at (46.5,-25.5) {$\ominus$};
\node at (49.5,-22.5) {$\oplus$};
\node at (52.5,-19.5) {$\oplus$};
\node at (55.5,-16.5) {\red{$\ominus$}};
\node at (58.5,-13.5) {$\ominus$};
\node at (55.5,-10.5) {$\oplus$};
\node at (49.5,-28.5) {$\oplus$};
\node at (55.5,-22.5) {$\ominus$};
\node at (58.5,-25.5) {$\oplus$};
\node at (61.5,-28.5) {\red{$\ominus$}};
%label
%caption
\node at (8,-5.5) {Case $I'$-1};
\node at (8,-33) {Case $I'$-2};
\node at (30,-5.5) {Case $I'$-3};
\node at (30,-33) {Case $I'$-4};
\node at (51,-5.5) {Case $II'$-1};
\node at (51,-33) {Case $II'$-2};
%N
\node at (37.2,13.3) {$L$};
\node at (37.2,-11.8) {$L$};
\node at (59.2,16.2) {$L$};
\node at (59.2,-11.8) {$L$};
%Case I
\node at (4.8,7.2) {$D$};
%\node at (8,10) {$L_2$};
\node at (10.6,13) {$D^*$};
%Case II
\node at (4.8,-17.8) {$D$};
%\node at (5,-18) {$L_1$};
%\node at (8,-15) {$L_2$};
\node at (10.8,-12) {$D^*$};
%Case III
\node at (26.5,4) {$D$};
\node at (29.8,7) {$D^*$};
%Case IV
\node at (26.8,-20.8) {$D$};
%\node at (27,-21) {$L_1$};
\node at (29.8,-17.8) {$D^*$};
%Case V
\node at (54.8,13.2) {$D$};
\node at (48.2,6.8) {$D^*$};
%\node at (55,13) {$N_*$};
%Case VI
\node at (48.8,-20.8) {$D^*$};
\node at (54.8,-14.8) {$D$};
%\node at (55,-15) {$N_*$};
\end{tikzpicture}
\caption{{\bf1.} Six cases for $\Psi^{-1}$ corresponding to Fig.~\ref{findadj}}
\label{findL}
\end{figure}
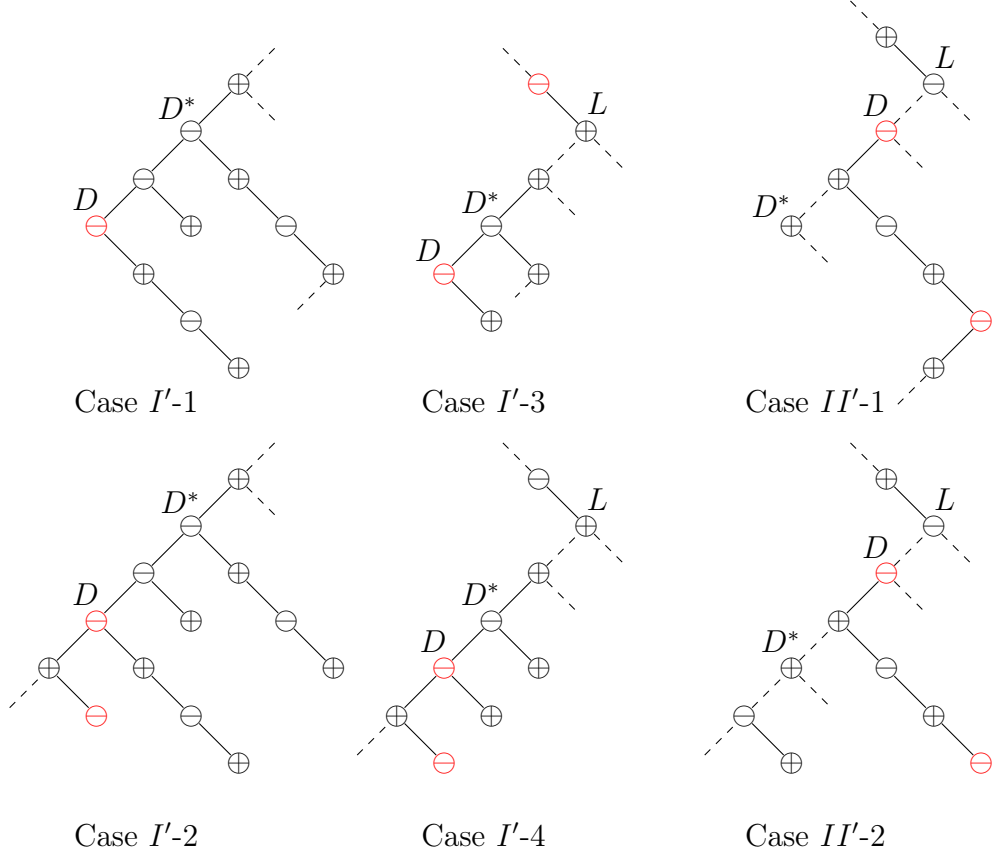

\begin{Def}[Adjoints for bad even $\ominus$-chains]\label{L}
Given a di-sk tree $T\in \D\T^{1}_{n,k}\backslash \D\T^{2}_{n,k}$, for each bad even $\ominus$-chain $D$ of $T$, we find a unique chain $D^*$ of $T$ at the same level, called the {\em adjoint of $D$}, according to the following two cases. We also denote $L$ the last node at the same level as the head of $D$.
\begin{itemize}
\item[Case $I'$:] $L$ is the root or a non-root labelled $\oplus$. We scan all the nodes from the head of $D$ to $L$, stop when we encounter a $\oplus$ or when we reach $L$, choose the last $\ominus$-chain as $D^*$. In this case, $D^*$ may equal $D$. Note that all the chains between $D$ and $D^*$ are even $\ominus$-chains. 
\item[Case $II'$:] $L$ is a non-root labelled $\ominus$. We scan all the nodes at the same level and before the head of $D$, stop when we encounter a $\ominus$ or when we reach the first node at this level, choose the last $\oplus$-chain as $D^*$. Note that all the chains between $D$ and $D^*$ are $\oplus$-chains.
\end{itemize}
%\begin{itemize}
%\item[(s-a):]  $D$ is an even $\ominus$-chain whose head is the first node of $T$ (see Case 1 in Fig.~\ref{findL}). 
%\item[(s-b):] $D$ is an even $\ominus$-chain whose head, together with the parent of $L$, form a pair of consecutive $\ominus$ (see Case 3 in Fig.~\ref{findL}). In this case $L$ must be a non-root labelled $\oplus$.
%\item[(s-c):] $D$ is an even $\ominus$-chain whose head (labelled $\ominus$), together with the tail of the chain just before it, form a pair of consecutive $\ominus$, and $L$ is the root or a non-root labelled $\oplus$ (see Cases 2,4 in Fig.~\ref{findL}).
%\item[(s-d):] $D$ is an even $\ominus$-chain whose head (labelled $\ominus$), together with the tail of the chain just before it, form a pair of consecutive $\ominus$, and $L$ is a non-root labelled $\ominus$ (see Cases 5,6 in Fig.~\ref{findL}).
%\end{itemize}
%When $D$ is in situations (s-a), (s-b) and (s-c), 
%
%
% When $D$ is in situation (s-d),  Both $D$ and its adjoint $D^*$ have been marked in Fig.~\ref{findL}.
\end{Def}

%%%%%example of bijection%%%%%%%%%%%%
\begin{figure}[]
\begin{tikzpicture}[scale=0.21]
%plot the trees
%2\1
%level
\draw[-] (8,1) to (10,3);
\draw[-] (11,4) to (13,6);
\draw[-] (14,7) to (16,9);
\draw[-] (17,10) to (19,12);
\draw[-] (20,13) to (22,15);
\draw[-] (14,19) to (16,21);
\draw[-] (17,22) to (19,24);
\draw[-] (20,25) to (22,27);
\draw[-] (2,13) to (4,15);
\draw[-] (5,16) to (7,18);
\draw[-] (2,25) to (4,27);
\draw[-] (5,28) to (7,30);
\draw[-] (8,31) to (10,33);
\draw[-] (11,34) to (13,36);
\draw[-] (14,37) to (16,39);
\draw[-] (17,40) to (19,42);
%chain
\draw[-] (11,3) to (13,1);
\draw[-] (17,9) to (19,7);
\draw[-] (20,6) to (22,4);
\draw[-] (20,12) to (22,10);
\draw[-] (17,21) to (19,19);
\draw[-] (20,18) to (22,16);
\draw[-] (20,24) to (22,22);
\draw[-] (2,12) to (4,10);
\draw[-] (5,9) to (7,7);
\draw[-] (5,15) to (7,13);
\draw[-] (8,12) to (10,10);
\draw[-] (2,24) to (4,22);
\draw[-] (5,21) to (7,19);
\draw[-] (8,18) to (10,16);
\draw[-] (11,15) to (13,13);
\draw[-] (5,27) to (7,25);
\draw[-] (8,30) to (10,28);
\draw[-] (11,27) to (13,25);
\draw[-] (11,33) to (13,31);
\draw[-] (14,30) to (16,28);
\draw[-] (14,36) to (16,34);
\draw[-] (17,33) to (19,31);
\draw[-] (20,30) to (22,28);
%1\2
%level
\draw[-] (32,25) to (34,27);
\draw[-] (35,28) to (37,30);
\draw[-] (38,31) to (40,33);
\draw[-] (41,34) to (43,36);
\draw[-] (44,37) to (46,39);
\draw[-] (32,13) to (34,15);
\draw[-] (35,16) to (37,18);
\draw[-] (50,19) to (52,21);
\draw[-] (53,22) to (55,24);
\draw[-] (44,1) to (46,3);
\draw[-] (47,4) to (49,6);
\draw[-] (50,7) to (52,9);
\draw[-] (53,10) to (55,12);
%chain
\draw[-] (32,12) to (34,10);
\draw[-] (35,15) to (37,13);
\draw[-] (38,12) to (40,10);
\draw[-] (41,9) to (43,7);
\draw[-] (32,24) to (34,22);
\draw[-] (35,21) to (37,19);
\draw[-] (38,18) to (40,16);
\draw[-] (35,27) to (37,25);
\draw[-] (38,30) to (40,28);
\draw[-] (41,27) to (43,25);
\draw[-] (44,24) to (46,22);
\draw[-] (41,33) to (43,31);
\draw[-] (44,30) to (46,28);
\draw[-] (44,42) to (46,40);
\draw[-] (44,36) to (46,34);
\draw[-] (47,33) to (49,31);
\draw[-] (50,30) to (55,25);
\draw[-] (53,21) to (55,19);
\draw[-] (50,18) to (52,16);
\draw[-] (53,15) to (55,13);
\draw[-] (44,0) to (46,-2);
\draw[-] (47,3) to (49,1);
\draw[-] (50,6) to (52,4);
\draw[-] (53,3) to (55,1);
\draw[-] (53,9) to (55,7);
\draw[-] (56,12) to (58,10);

%put nodes
%2\1
%level
\node at (7.5,0.5) {$\oplus$};
\node at (10.5,3.5) {$\ominus$};
\node at (13.5,6.5) {$\ominus$};
\node at (16.5,9.5) {$\oplus$};
\node at (19.5,12.5) {$\oplus$};
\node at (22.5,15.5) {$\oplus$};
\node at (13.5,18.5) {$\ominus$};
\node at (16.5,21.5) {$\oplus$};
\node at (19.5,24.5) {$\ominus$};
\node at (22.5,27.5) {$\ominus$};
\node at (1.5,12.5) {$\oplus$};
\node at (4.5,15.5) {$\ominus$};
\node at (7.5,18.5) {$\oplus$};
\node at (1.5,24.5) {$\oplus$};
\node at (4.5,27.5) {$\ominus$};
\node at (7.5,30.5) {$\ominus$};
\node at (10.5,33.5) {$\oplus$};
\node at (13.5,36.5) {$\oplus$};
\node at (16.5,39.5) {$\oplus$};
\node at (19.5,42.5) {$\ominus$};
\node at (4.5,21.5) {$\ominus$};
\node at (7.5,24.5) {$\oplus$};
\node at (10.5,27.5) {$\oplus$};
\node at (13.5,30.5) {$\ominus$};
\node at (16.5,33.5) {$\ominus$};
\node at (13.5,24.5) {$\ominus$};
\node at (16.5,27.5) {$\oplus$};
\node at (19.5,30.5) {$\oplus$};
\node at (4.5,9.5) {$\ominus$};
\node at (7.5,12.5) {$\oplus$};
\node at (10.5,15.5) {$\ominus$};
\node at (7.5,6.5) {$\oplus$};
\node at (10.5,9.5) {$\ominus$};
\node at (13.5,12.5) {$\oplus$};
\node at (19.5,18.5) {$\ominus$};
\node at (22.5,21.5) {$\oplus$};
\node at (13.5,0.5) {$\oplus$};
\node at (19.5,6.5) {$\ominus$};
\node at (22.5,9.5) {$\ominus$};
\node at (22.5,3.5) {$\oplus$};

%1\2
\node at (43.5,0.5) {$\ominus$};
\node at (46.5,3.5) {$\ominus$};
\node at (49.5,6.5) {$\oplus$};
\node at (52.5,9.5) {$\oplus$};
\node at (55.5,12.5) {$\oplus$};
\node at (46.5,-2.5) {$\oplus$};
\node at (49.5,0.5) {$\oplus$};
\node at (52.5,3.5) {$\ominus$};
\node at (55.5,6.5) {$\ominus$};
\node at (58.5,9.5) {$\ominus$};
\node at (55.5,0.5) {$\oplus$};
\node at (52.5,15.5) {$\ominus$};
\node at (55.5,18.5) {$\oplus$};
\node at (49.5,18.5) {$\oplus$};
\node at (52.5,21.5) {$\ominus$};
\node at (55.5,24.5) {$\ominus$};
\node at (43.5,6.5) {$\oplus$};
\node at (40.5,9.5) {$\ominus$};
\node at (34.5,9.5) {$\ominus$};
\node at (37.5,12.5) {$\oplus$};
\node at (40.5,15.5) {$\ominus$};
\node at (46.5,21.5) {$\oplus$};
\node at (31.5,12.5) {$\oplus$};
\node at (34.5,15.5) {$\ominus$};
\node at (37.5,18.5) {$\oplus$};
\node at (43.5,24.5) {$\ominus$};
\node at (46.5,27.5) {$\oplus$};
\node at (49.5,30.5) {$\oplus$};
\node at (34.5,21.5) {$\ominus$};
\node at (37.5,24.5) {$\oplus$};
\node at (40.5,27.5) {$\oplus$};
\node at (43.5,30.5) {$\ominus$};
\node at (46.5,33.5) {$\ominus$};
\node at (31.5,24.5) {$\oplus$};
\node at (34.5,27.5) {$\ominus$};
\node at (37.5,30.5) {$\ominus$};
\node at (40.5,33.5) {$\oplus$};
\node at (43.5,36.5) {$\oplus$};
\node at (46.5,39.5) {$\oplus$};
\node at (43.5,42.5) {$\ominus$};

%label
%caption
\node at (12,-6) {$T\in \D\T^{2}\backslash \D\T^{1}$};
\node at (44,-6) {$\Psi(T)\in \D\T^{1}\backslash \D\T^{2}$};
%2\1
\node at (18.8,44.2) {$C_5$};
\node at (15.8,41.2) {$C_5^*$};
\node at (6.8,32.2) {$C_2$};
\node at (0.8,26.2) {$C_2^*$};
\node at (3.8,17.2) {$C_1$};
\node at (0.8,14.2) {$C_1^*$};
\node at (12.8,8.2) {$C_4$};
\node at (6.8,2.2) {$C_4^*$};
\node at (12.8,20.2) {$C_3$};
\node at (15.8,23.2) {$C_3^*$};
%1\2
\node at (40,44) {$D_5^*=D_5$};
\node at (33.7,29.2) {$D_2$};
\node at (36.8,32.3) {$D_2^*$};
\node at (45.8,5.2) {$D_4^*$};
\node at (42.3,2) {$D_4$};

\node at (32,17.2) {$D_1^*=D_1$};

\node at (47.5,19.4) {$D^*_3$};
\node at (50.5,22.4) {$D_3$};
%map
\node at (27,25) {$\Psi \atop \longrightarrow$};
\node at (27,21) {$\longleftarrow\atop \Psi^{-1}$};
\end{tikzpicture}
\caption{An example of $\Psi$.}
\label{bijexample}
\end{figure}
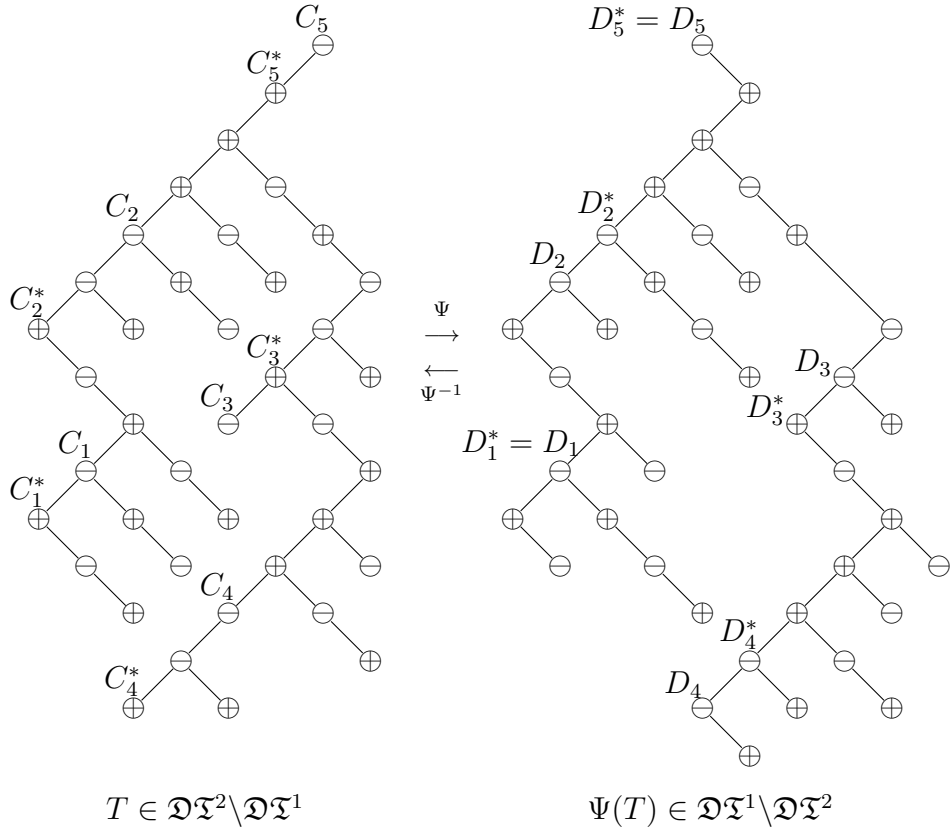

 It is routine to check that no two pairs $(D_1,D_1^*)$ and $(D_2,D_2^*)$ can intertwine with each other.
We can now construct $\Psi^{-1}$ as follows.  Take any $T\in \D\T^{1}_{n,k}\backslash \D\T^{2}_{n,k}$. Find all the adjoints $D^*$ for all bad even $\ominus$-chains $D$. We perform the following operations and denote the resulting tree as $\Psi^{-1}(T)$.
\begin{itemize}
\item $D\ge D^*$ and $D$ is the first chain at this level (see Cases $I'$-1, $I'$-3 in Fig.~\ref{findL}). We cut off the tail of $D^*$ (labelled $\oplus$), together with its left subtree if any, and attach it to the head of $D$ from left, creating a new first chain at this level, and leaving $D^*$ as an odd $\ominus$-chain.
\item $D\ge D^*$ and $D$ is not the first chain at this level (see Cases $I'$-2, $I'$-4 in Fig.~\ref{findL}). We cut off the tail of $D^*$ (labelled $\oplus$), together with its left subtree if any, and attach it (from right) to the tail of the chain just before $D$, leaving $D^*$ as an odd $\ominus$-chain.
\item $D^*>D$ and $D^*$ is the first chain at this level (see Case $II'$-1 in Fig.~\ref{findL}). We cut off the node just before the head of $D$ (labelled $\ominus$), together with its left subtree if any, and attach it to the head of $D^*$ from left, creating a new first chain at this level, which is a one-node $\ominus$-chain.
\item $D^*>D$ and $D^*$ is not the first chain at this level (see Case $II'$-2 in Fig.~\ref{findL}). We cut off the node just before the head of $D$ (labelled $\ominus$), together with its left subtree if any, and attach it (from right) to the tail of the chain just before $D^*$, making it an odd $\ominus$-chain.
\end{itemize}

In each case, either the first $\ominus$-node or a pair of consecutive $\ominus$ will disappear, but a new odd $\ominus$-chain will arise, so $\Psi^{-1}(T)$ is in $\D\T^{2}_{n,k}\backslash \D\T^{1}_{n,k}$ indeed, and $\Psi^{-1}$ is well defined. It is fairly easy to check in all 6 cases (compare Fig.~\ref{findL} with Fig.~\ref{findadj}) that $\Psi^{-1}$ is really the inverse of $\Psi$, as desired.
Such a delicate construction deserves some examples, and we offer one in Fig.~\ref{bijexample}.

\begin{remark}
Our bijection $\Psi$ between $\D\T_{n,k}^{2}$ and $\D\T_{n,k}^{1}$ is a bit involved. It would be appealing to find a direct group action on di-sk trees $\D\T_n$ such that each orbit has exactly one element in $\D\T_{n,k}^{2}$. The {\em tree Eulerian polynomial} $T_n(t)=\prod_{i=1}^{n-1}(n-i+it)$ is the descent polynomial of labeled  rooted trees on $[n]$. Recently, Gonz\'alez D'Le\'on~\cite{go} found several  interpretations  for the $\gamma$-coefficients of  $T_n(t)$ in terms of other combinatorial models but not in terms of labeled rooted trees.
Another  open problem that has the same flavor is to 
define an action on labeled rooted trees which results in  an interpretation for the corresponding $\gamma$-coefficients.  
\end{remark}
%%%%%%%%%%%%%%%%%%%%%%%%%%%%
\section{Spiral property for $D_n(t)$}
\label{sec:spiral}
%%%%%%%%%%%%%%%%%%%%%%%%%%%%%%%

 To prove the unimodality of $D_n(t)$, we shall apply the following formula of 
 D\'esarm\'enien--Foata~\cite{df} and Gessel--Reutenauer~\cite{ge}:
\begin{equation}\label{gessel}
 \sum_{n\geq2}\frac{D_n(t)}{(1-t)^{n+1}}z^n=\sum_{r\geq1}\frac{t^{r-1}}{1-rz}(1-z)^r.
\end{equation}
 
 \begin{lemma}
 The polynomial $D_n(t)$ satisfies the following recurrence relation:
 \begin{equation}\label{der}
 D_n(t)=(-1)^nt^{n-1}+(1+(n-1)t)D_{n-1}(t)+t(1-t)D'_{n-1}(t).
 \end{equation}
 Equivalently, 
 \begin{equation}\label{recu:der}
 d_{n,k}=
 \begin{cases}
 \,\,1, \qquad\qquad\qquad\qquad\text{if $n$ is even and $k=n-1$;}\\
 \,\,0,\,\,\qquad\qquad\qquad\qquad \text{if $n$ is odd and $k=n-1$;}\\
 \,\,(k+1)d_{n-1,k}+(n-k)d_{n-1,k-1},\,\,\quad\text{if $k\neq n-1$}.
 \end{cases}
 \end{equation}
 \end{lemma}
 
 \begin{proof}
 Extracting the coefficient of $z^n$ on both sides of~\eqref{gessel} gives
 \begin{equation}\label{euler}
 \frac{D_n(t)}{(1-t)^{n+1}}=\sum_{r\geq1}t^{r-1}T_r(n),%\left(\sum_{k=0}^{n\wedge r}(-1)^k{r\choose k}r^{n-k}\right),
 \end{equation}
 where %$n\wedge r=\min\{n,r\}$.and 
$T_r(n)=\sum_{k=0}^{\min\{n,r\}}(-1)^k{r\choose k}r^{n-k}$. Since
 $$
 T_r(n)=
 \begin{cases}
 rT_r(n-1), &\text{if $1\leq r\leq n-1$};\\
 rT_r(n-1)+(-1)^n{r\choose n}, &\text{otherwise},
 \end{cases}
 $$
we derive  from~\eqref{euler} that 
 \begin{align*}
 t\left(\frac{tD_{n-1}(t)}{(1-t)^n}\right)'&=\sum_{r\geq1}t^rrT_r(n-1)\\
 &=\sum_{r\geq1}t^rT_r(n)-\sum_{r\geq n}t^r(-1)^n{r\choose n}\\
 &=\frac{tD_n(t)-(-t)^n}{(1-t)^{n+1}}.
 \end{align*}
 After simplifying we get~\eqref{der}.
 \end{proof}
 \begin{remark} Let $d_n=D_n(1)$ be the cardinality of  $\D_n$. 
 When $t=1$, Eq.~\eqref{der} reduces to the well-known recurrence relation 
 %\begin{equation*}\label{der:number}
$ d_n=(-1)^n+nd_{n-1}$.
 %\end{equation*}
%for the numbers of derangements $d_n=\#\D_n$. 
It is also reminiscent of the recurrence 
 \begin{equation}\label{rec:euler}
 A_n(t)=(1+(n-1)t)A_{n-1}(t)+t(1-t)A'_{n-1}(t)
 \end{equation}
for the Eulerian polynomials.
 \end{remark}

 A {\em desarrangement} is a permutation whose first ascent is  even, where an index $i\in[n]$ is an {\em ascent} of $\pi\in\S_n$ if $\pi_i<\pi_{i+1}$ (by convention $\pi_{n+1}=+\infty$). For example, $653\red{\bf2}41$ is a desarrangement but $32\red{\bf1}564$ is not.  
 Let  $\E_n$ be the set of all desarrangements in $\S_n$.
 
 \begin{proof}[A bijective proof of~\eqref{recu:der}] 
 By a result of D\'esarm\'enien and Wachs~\cite[Corollary~3.3]{dw}) 
% proved that 
% \begin{equation}\label{desar}
% \sum_{\pi\in\D_n}t^{\des(\pi)}=\sum_{\pi\in\E_n}t^{\ides(\pi)},
% \end{equation}
% We will use the interpretation 
we have 
 $$
 d_{n,k}=\{\pi\in\E_n:\ides(\pi)=k\}.
 $$

We say that an index $i$, $1\leq i\leq n-1$, is an {\em inverse descent} of $\pi\in\S_n$ if $i+1$ appears to the left of $i$ in $\pi$. Clearly, the number of inverse descents of $\pi$ is $\ides(\pi)$. When $n$ is even, the only desarrangement in $\E_n$ with $n-1$ inverse descents is $n(n-1)\cdots 21$, so $d_{n,n-1}=1$ in this case. In the $n$ odd case, there is not desarrangement of length $n$ with $n-1$ inverse descents and $d_{n,n-1}=0$ follows. In the following, we can assume that $1\leq k<n-1$.

%Let $\S_{n-1}\times[n]:=\{(\pi,j):\pi\in\S_{n-1},j\in[n]\}$. 
There is a natural bijection  from $\S_{n-1}\times[n]$ to $\S_n$ defined by 
\begin{equation}\label{mapping}
(\pi,j)\mapsto\sigma=\sigma_1\cdots\sigma_n, 
\end{equation}
where $\sigma_n=j$ and for $ i\in[n-1]$, $\sigma_i=\pi_i+1$ if $\pi_i\geq j$, otherwise $\sigma_i=\pi_i$.
It is routine to check that 
$$
\ides(\sigma)=
\begin{cases}
\,\,\ides(\pi),\qquad&\text{if $j-1$ is an inverse descent of $\pi$;}\\
\,\,\ides(\pi)+1,\qquad&\text{otherwise}.
\end{cases}
$$
Recurrence relation~\eqref{recu:der} then follows from this property and the fact that in~\eqref{mapping} if $\sigma$ is a desarrangement in $\E_n$ with $k$ ($k<n-1$) inverse descents then $\pi$ is a desarrangement.
 \end{proof}
 
 From~\eqref{der} we can readily deduce that $\deg(D_{2n+1}(t))=2n-1$ and $D_{2n}(t)$ is a monic polynomial of degree $2n-1$. Moreover, the coefficient of $t$ in $D_n(t)$ is $2^{n-2}$.
  
 \begin{proof}[Proof of Theorem~\ref{thm:spiral}]
  It is easy to check that statement~\eqref{spiral} is true for $n\leq 3$. We proceed to prove the statement by induction on $n$ using recurrence~\eqref{recu:der}.
 
 Suppose that $m\geq4$ and statement~\eqref{spiral} is true for $n=m-1$. We first show that $d_{2m,2m-k}<d_{2m,k}< d_{2m,2m-k-1}$ for $1\leq k\leq m-1$. By the recurrence relation~\eqref{recu:der} for $d_{n,k}$, we have
 \begin{equation}\label{eq:1}
 d_{2m,2m-k}=(2m-k+1)d_{2m-1,2m-k}+kd_{2m-1,2m-k-1}
 \end{equation}
 if $k\neq1$ and
  \begin{equation}\label{eq:2}
 d_{2m,k}=(k+1)d_{2m-1,k}+(2m-k)d_{2m-1,k-1}
 \end{equation}
  \begin{equation}\label{eq:3}
 d_{2m,2m-k-1}=(2m-k)d_{2m-1,2m-k-1}+(k+1)d_{2m-1,2m-k-2}.
 \end{equation}
 Clearly, $d_{2m,2m-1}=1<2^{2m-2}=d_{2m,1}$. It follows from~\eqref{eq:1} and~\eqref{eq:2} that, for $k\geq2$, 
 \begin{align*}
 d_{2m,k}-d_{2m,2m-k}&=(2m-k+1)(d_{2m-1,k-1}-d_{2m-1,2m-k})\\
 &\quad+k(d_{2m-1,k}-d_{2m-1,2m-k-1})+(d_{2m-1,k}-d_{2m-1,k-1}).
 \end{align*}
 By the inductive hypothesis, the difference in every parenthesis in the above expression is positive, which implies that $d_{2m,k}>d_{2m,2m-k}$. Similarly, by~\eqref{eq:2} and~\eqref{eq:3} we have
 \begin{align*}
 d_{2m,2m-k-1}-d_{2m,k}&=(2m-k)(d_{2m-1,2m-k-1}-d_{2m-1,k-1})\\
 &\quad+(k+1)(d_{2m-1,2m-k-2}-d_{2m-1,k}).
 \end{align*}
 Again, by the inductive hypothesis, we deduce that $d_{2m,2m-k-1}>d_{2m,k}$. This completes the proof of the first part of statement~\eqref{spiral} for $n=m$. 
 It remains to show the second part of statement~\eqref{spiral} for $n=m$, which is omitted due to the  similarity. This completes the proof of the theorem by induction. 
 \end{proof}

%%%%%%%%%%%%%%%%%%%% 
\section{Concluding remarks and open problems}\label{sec:gessel}
%%%%%%%%%%%%%%%%%%%%

 The combinatorial interpretation for $\gamma_{n,k}^S$ that we established in Theorem~\ref{gamma:s} nicely parallels those for $\gamma_{n,k}^A$ and $\gamma_{n,k}^N$, and note that $\S_n(231)\subseteq\S_n(2413,3142)\subseteq\S_n$. This in particular will give as by-product: the descent polynomials on permutations that contain at least one of the patterns $(2413,3142)$ are also $\gamma$-positive. Similar result holds for  $\S_n(231)$, $\S_n\backslash\S_n(231)$ and $\S_n(2413,3142)\backslash\S_n(231)$. This observation raises a natural question: are there any other subsets  of $\S_n$ that enjoy the same property? The reader is referred to the book of Kitaev~\cite{ki} for other interesting pattern avoiding classes of permutations.
%  Along this line, many nice results concerning the classes of permutations avoiding two patterns of length $4$ are obtained recently in~\cite{kim,lin2}.

Theorem~\ref{gamma:s} was partially motivated  by the second author's recent proof~\cite{lin} of 
a conjecture of 
Gessel \cite{br},  which  states that for $n\geq 1$, there exist \emph{nonnegative} integers $\gamma_{n,i,j}$, $0\leq i,j$, $j+2i\leq n-1$, such that 
\begin{equation}
\label{tEu}
\sum_{\sigma\in \S_n}s^{\ides(\sigma)}t^{\des(\sigma)}
= \sum_{i,j\geq 0}\gamma_{n,i,j}(st)^i(1+st)^j(s+t)^{n-1-j-2i},
\end{equation}
where $\ides(\sigma)$ denotes the number of descents of $\sigma^{-1}$.
Note that we recover Eulerian polynomial and its $\gamma$-expansion~\eqref{Eulerian} by 
setting $s=1$ in~\eqref{tEu}. 
In an effort to find combinatorial interpretation of $\gamma_{n,i,j}$ in~\eqref{tEu}, we restricted our attention to the terms without $s+t$, whose coefficients are $\gamma_{n,i,n-1-2i}$. This  leads us naturally to consider the operations $\oplus$ and $\ominus$, as 
 \begin{align*}
\des(\pi\oplus\sigma)=\des(\pi)+\des(\sigma),\qquad & \ides(\pi\oplus\sigma)=\ides(\pi)+\ides(\sigma);\\
\des(\pi\ominus\sigma)=\des(\pi)+\des(\sigma)+1,\qquad & \ides(\pi\ominus\sigma)=\ides(\pi)+\ides(\sigma)+1.
\end{align*}
In view of Proposition~\ref{desides}, we have 
\begin{align}\label{des-ides}
\des(\pi)=\ides(\pi)
\end{align}
for each $\pi\in\S_n(2413,3142)$.   Another  interesting 
class of permutations satisfying \eqref{des-ides}  is the set of  {\em involutions}.  A conjecture on the $\gamma$-positivity of the descent polynomial on involutions 
in $\S_n$ 
 was first made by Guo and Zeng~\cite{gz} and   is  still open.
 
 %However, the unimodality of descent polynomial on derangements has long been overlooked. 

It follows from Theorem~\ref{gamma:s} that 
$\gamma^A_{n,k}\geq  \gamma^S_{n,k}$,
which is hard to be proved by analysis using generating functions.
Originally, it is the attempt to prove this inequality that inspires us to find the interpretation of $\gamma^S_{n,k}$ in Theorem~\ref{gamma:s}. 
 Regarding Gessel's conjecture,  the following stronger inequality seems also true.

\begin{conjecture}
Let $\gamma_{n,i,j}$ be defined by~\eqref{tEu}. Then,
$$
\gamma_{n,k,n-1-2k}\geq \gamma^S_{n,k}.
$$
\end{conjecture}

Let $\widetilde{D}_n(t)$ be the descent polynomial on $\S_n\backslash\D_n$. It follows from~\eqref{der} and~\eqref{rec:euler} that 
 \begin{equation*}
\widetilde{D}_n(t)=(-t)^{n-1}+(1+(n-1)t)\widetilde{D}_{n-1}(t)+t(1-t)\widetilde{D}^{\prime}_{n-1}(t),
 \end{equation*}
since $\widetilde{D}_n(t)=A_n(t)-D_n(t)$. By similar discussion as in the proof of Theorem~\ref{thm:spiral}, we can show that $\widetilde{D}_n(t)$ also has the spiral property, which implies its unimodality.

Finally, the two descent polynomials $S_n(t)$ and $D_n(t)$ seem to be 
real-rooted based on our computational experiments. We pose this as a conjecture for further investigation.
\begin{conjecture} 
The descent polynomials $S_n(t)$ and $D_n(t)$ are real-rooted for each $n\geq2$.
\end{conjecture}
 
%%%%%%%%%%%%%%%%%%%%%%%%%% 
\section*{Acknowledgement} 
%%%%%%%%%%%%%%%%%%%%%%%%%%
We thank an anonymous referee for her/his useful suggestions which improved the exposition of our main bijection.  This work was done during the 
2015 NIMS Thematic Program on Combinatorics 
at CAMP (Center for Applications of Mathematical Principles), Daejeon.
We  would like to thank CAMP for providing excellent working condition
during this program. 

The first author's research was partially supported by the National Science Foundation of China grant 11501061 and the Fundamental Research Funds for the Central Universities No.~CQDXWL-2014-Z004. The second author's research was partially supported by the National Science Foundation of China grant 11501244.

%%%%%%%%%%%%%%

\end{document}